\documentclass[11pt, a4paper]{article}
\usepackage{amsmath,amsthm,amssymb,amsfonts,esint}
\usepackage{a4wide}
\pdfoutput=1

\usepackage[utf8]{inputenc}
\usepackage[T1]{fontenc}
\usepackage{lmodern}
\usepackage[english]{babel}
\usepackage[colorlinks=true, pdfstartview=FitV, linkcolor=blue, citecolor=blue, urlcolor=blue,pagebackref=false]{hyperref}
\usepackage{mathtools}
\usepackage{enumerate}

\usepackage[backend=bibtex,firstinits,maxnames=4,style=alphabetic,isbn=false, date=year, doi=false, eprint= true, url= false]{biblatex}

\usepackage{xcolor}

\usepackage{microtype}

\newtheorem{thm}{Theorem}[section]
\newtheorem{prop}[thm]{Proposition}
\newtheorem{lem}[thm]{Lemma}
\newtheorem{cor}[thm]{Corollary}
\theoremstyle{remark}
\newtheorem{rem}[thm]{Remark}
\theoremstyle{definition}

\renewcommand{\leq}{\leqslant}
\renewcommand{\geq}{\geqslant}

\renewcommand{\subset}{\subseteq}

\newcommand{\N}{\mathbb{N}}

\newcommand{\1}{\mathbf{1}}
\newcommand{\R}{\mathbb{R}}
\newcommand{\IR}{\mathbb{R}}

\newcommand{\dd}{\, \mathrm{d}}

\newcommand{\eff}{\mathrm{eff}}

\newcommand{\Hdot}{\dot{H}}
\newcommand{\dmin}{d_{\min}}

\DeclareMathOperator{\Sym}{Sym}
\DeclareMathOperator{\curl}{curl}

\DeclareMathOperator{\Id}{Id}

\DeclareMathOperator*{\dv}{div}

\DeclareMathOperator*{\supp}{supp}

\numberwithin{equation}{section}
\mathtoolsset{showonlyrefs}

\usepackage{authblk}
\newcommand*{\email}[1]{%
    \normalsize\href{mailto:#1}{#1}\par
    }
    
\title{Convergence of the Method of Reflections for Particle Suspensions in Stokes Flows}
\author{Richard M. H\"ofer \thanks{Present affiliation: Insitut de Mathématiques de Jussieu - Paris Rive Gauche, Université de Paris, 8 Place Aurélie Nemours, 75013 Paris, France \email{hoefer@imj-prg.fr}}}
    
\affil{Institute for Applied Mathematics, University of Bonn, Endenicher Allee 60, 53115 Bonn, Germany}
\affil{Institut de Mathématiques, Université de Bordeaux, 351 Cours de la Libération, 33400 Talence, France}

\bibliography{Hoefer.bib}

\begin{document}

\maketitle

\begin{abstract}
We study the convergence of the method of reflections for the Stokes equations in domains perforated by countably many spherical particles with boundary conditions typical for the suspension of rigid particles. We prove that a relaxed version of the method is always convergent in $\dot H^1$ under a mild separation condition on the particles. Moreover, we prove optimal convergence rates of the method in $\dot W^{1,q}$, $1 < q < \infty$ 
and in $L^\infty$ in terms of the particle volume fraction under a stronger separation condition of the particles.

\par\vskip\baselineskip\noindent
\end{abstract}

\section{Introduction}

We consider a perforation of $\R^3$ by disjoint spherical particles with radii $R_i$ located at positions $X_i$, $i \in I$, where $I$ is a finite or countable index set. 
We write $B_i := B_{R_i}(X_i)$ and 
\begin{align}
	K := \bigcup_i \overline{B_i} 
\end{align}
for the set occupied by the particles.


%

We are interested in solving the Stokes problem
\begin{align} 
	- \mu \Delta v + \nabla p = f \quad \text{in } \R^3 \setminus K, \label{eq:Stokes.perforated} \\
	e v = 0 \quad \text{in } K \label{eq:const} \\
	\int_{\partial B_i} \sigma[v,p] n = \int_{B_i} f \quad \text{for all } i \in I. \label{eq:forcefree} \\
	\int_{\partial B_i} (x - X_i) \times \sigma[v,p] n  = \int_{B_i} (x- X_i) \times f \quad \text{for all } i \in I. \label{eq:torquefree} \\
\end{align}
Here,  $e v = \frac 1 2 (\nabla v + \nabla v^T)$ is the symmetric gradient. By a standard rigidity result, \eqref{eq:const} means that the fluid velocity $v$ (extended to the interior of the particles) is a rigid body motion inside of each particle.
This condition is complemented by \eqref{eq:forcefree} and \eqref{eq:torquefree} where
$\sigma[u] = 2 \mu e u - p \Id$ is the fluid stress. These two constraints prescribe the total force and torque acting on each particle.

We remark, that all the results presented in this paper continue to hold for the system corresponding to \eqref{eq:Stokes.perforated} - \eqref{eq:torquefree}
when the Stokes equations are replaced by the Poisson equation.

\medskip

The method of reflection yields a formal series expansion of the solution to certain (linear) PDEs like \eqref{eq:Stokes.perforated} - \eqref{eq:torquefree}
in domains like $\R^3 \setminus K$, where the boundary of the domain is decomposed into several (possibly infinitely many) disjoint boundaries. 
This series representation is of the form
\begin{align} \label{eq:MOR.series}
	v = u - \sum_{i_1 \in I} Q_{i_1} u + \sum_{i_2 \in I} \sum_{\substack {i_1 \in I \\ i_1 \neq i_2}} Q_{i_2} Q_{i_1} u - 
	\sum_{i_3 \in I} \sum_{\substack {i_2 \in I \\ i_2 \neq i_3}} \sum_{\substack {i_1 \in I \\ i_1 \neq i_2}} Q_{i_3} Q_{i_2} Q_{i_1} u + \dots,
\end{align}
where $u$ is the solution to the PDE in the whole space and the operators $Q_i$ are certain solution operators which take into account only the $i$-th part of the boundary. 
%
%
This method has been systematically studied by Vel\'azquez and the author
in \cite{HoferVelazquez18} for the Poisson and Stokes equations with Dirichlet boundary conditions. 
We showed in \cite{HoferVelazquez18} how the method of reflections for Dirichlet boundary can be used to re-obtain classical homogenization results
for the Poisson and Stokes equations in perforated domains. 
The method has been used in \cite{FigariOrlandiTeta85,Rubinstein1986} to determine the fluctuations in related problems,
where the particles are randomly distributed.
We refer to \cite{CiaramellaGanderHalpernSalomon18} and \cite{HoferVelazquez18}
for a more general introduction, historical remarks and more references on the method of reflections and its applications.


\subsection{Applications of the method in sedimentation and the effective viscosity problems}
\label{sec:applications}

In several recent works, the method of reflections has been applied to system \eqref{eq:Stokes.perforated} - \eqref{eq:torquefree}
which is a typical model for the study of particle suspensions.
In particular, the method of reflections has been used for effective viscosity problem and the problem of (inertialess) particle sedimentation.

The effective viscosity problem aims at describing the phenomenon that the presence of the particles give rise to an 
effective viscosity $\mu_{\eff}$ in the homogenization limit of many small particles.
It is believed that an expansion of $\mu_\eff$ in terms of the particle volume fraction $\phi$ of the form 
\begin{align} \label{eq:effective.viscosity.expansion}
	\mu_{eff} = \mu + \frac{5}{2} \phi + \mu_2 \phi^2 + ...
\end{align}
holds, where $\frac{5}{2} \phi$ is known as the Einstein correction \cite{Ein06}.

On the other hand, the convergence of the method of reflections is expected at least in powers of
 \begin{align} \label{def:phi_0}
	\phi_0 := \frac{R_{\max}^3}{d_{\min}^3}.
\end{align} 
where $R_{\max}$ is the maximal particle radius and $d_{\min}$ the minimal particle distance.
More precisely, 
one expects the method of reflection to yield an approximation of the solution up to an error of order $o(\phi_0^k)$ by keeping the first $k+1$ terms in \eqref{eq:MOR.series}, i.e. by cutting the expansion after the term involving $k$ sums.
Under the (stringent) assumption $\phi \sim \phi_0$, this allows, at least formally, to determine the 
coefficients in the expansion for the effective viscosity in \eqref{eq:effective.viscosity.expansion}
by comparison with the expansion \eqref{eq:MOR.series}.

\medskip

If the fluid equations \eqref{eq:Stokes.perforated} - \eqref{eq:torquefree} are complemented by the equations of motion for the particles,
the problem becomes dynamical. A particularly physically relevant case is the dynamics of particle sedimentation, where the driving force is gravity acting on the particles. We emphasize that this case is included into the above system of equations \eqref{eq:Stokes.perforated} - \eqref{eq:torquefree} through a choice of $f$
such that $f = 0$ in $\R^3 \setminus K$ and $\int_{B_i} f = F_i$ where $F_i$ is the gravity acting on particle $B_i$.

In typical dynamical suspension problems such as the sedimentation of inertialess particles, the particle positions change according to $\dot X_i = v(X_i)$.
The method of reflections appears to be very helpful in this context since it provides the necessary pointwise information on the solution in contrast to classical energy methods.

\medskip

We give a brief summary about the results regarding the effective viscosity and sedimentation 
where the method of reflections has been used.

In \cite{NiethammerSchubert19}, Niethammer and Schubert used the method of reflections to rigorously establish the Einstein law, i.e. the expansion in \eqref{eq:effective.viscosity.expansion} up to order $o(\phi)$, for a suspension of spherical particles. Later, Hillairet and Wu \cite{HillairetWu19} proved a corresponding result for more general shapes of particles.
Moreover, Gerard-Varet, Hillairet and Mecherbet \cite{Gerard-VaretHillairet19, Gerard-VaretMecherbet20} were able to get insights on the second order correction in $\phi$ to the effective viscosity. For more results on the effective viscosity problem (not relying on the method of reflections), we refer the reader to \cite{HainesMazzucato12,Gerard-Varet19,DuerinckxGloria19, Gerard-VaretHoefer20}.

We also point out that in the recent works \cite{DuerinckxGloria20b,GerardVaret20}, the authors succeeded to go beyond the first order correction without the stringent assumption that the volume fraction scales like $\phi \sim \phi_0$ with $\phi_0$ as in \eqref{def:phi_0}. In this case, even if the method of reflections is converging, the $k$-th term in the series \eqref{eq:MOR.series} does in general not anymore correspond to the $k$-th order correction of the effective viscosity $\mu_k$ in \eqref{eq:effective.viscosity.expansion} for $k \geq 2$. Instead, \cite{DuerinckxGloria20b,GerardVaret20} rely on a cluster expansion.

\medskip
Regarding sedimentation problems, 
in \cite{JabinOtto04}, Jabin and Otto used a variant of the method of reflections to identify the dilute regime of inertialess particle sedimentation  where
the dynamics is close to the dynamics of isolated particles. Later, the author studied the corresponding (strong) interacting regime in \cite{Hofer18MeanField}, and proved convergence of the microscopic dynamics
for inertialess particle sedimentation
to the transport-Stokes system 
\begin{equation} \label{eq:Transport.Stokes}
\begin{aligned}
	\partial_t \rho + (u + \gamma^{-1} g) \cdot \nabla \rho = 0 \\
	-\Delta u + \nabla p = \rho g, \qquad \dv u = 0,
\end{aligned}
\end{equation}
where  $\rho(t,x)$ is the particle density, $g \in \R^3$ is the constant gravity and $\gamma$ is a parameter accounting for the particle interaction strength.
Using a related reflection method, quantitative convergence to \eqref{eq:Transport.Stokes} in terms of Wasserstein distances
has been proved by Mecherbet in \cite{Mecherbet18}. In \cite{Mecherbet19}, Mecherbet used similar methods for studying the sedimentation of close pairs of particles.

In the resent paper \cite{HoeferSchubert21}, the author and Schubert refined the aforementioned results taking into account the increase of the viscosity to first order in the particle volume fraction.

\medskip

We also mention the paper \cite{HillairetLacaveWu19}, where the method of reflections is applied do derive a  homogenization result for the Euler equations in a two-dimension perforated domain. Indeed, in this setting, the stream function satisfies a scalar form of problem \eqref{eq:Stokes.perforated}--\eqref{eq:torquefree}.

\subsection{Expected convergence rates and results of previous works} \label{sec:heuristics}

In \cite{HoferVelazquez18}, the method of reflections for the Dirichlet problem has been formulated in terms of orthogonal projections on the homogeneous Sobolev space 
$\dot H^1(\R^3)$. More precisely, using that the operators $Q_i$ in \eqref{eq:MOR.series} are orthogonal projections, the $k$-th order approximation of the method of reflections can be written as
\begin{align}
	v_k = (1- \sum_{i \in I} Q_i)^k u.
\end{align}
The analysis of the convergence thus boils done to the study of the operator $L = \sum_{i \in I} Q_i$.
For the Poisson equations with Dirichlet boundary conditions, the function $Q_i u$ is approximately given by
\begin{align}
	Q_i u(x) = \frac{(u)_i R_i}{|x - X_i|}, \qquad  \text{for } x \in \R^3 \setminus B_i
\end{align}
where $(u)_i$ is the average of $u$ over $\partial B_i$. A similar formula holds for the Stokes equations.
In particular, the operator $L$ is in general not well-defined if the particles are distributed everywhere in $\R^3$.
In \cite{HoferVelazquez18}, convergence of the method is therefore first shown for the screened Poisson equation (i.e. the operator $-\Delta + \lambda$, $\lambda > 0$),
if the particles are sufficiently well separated and the capacity density of the particles is sufficiently small. 
For the Poisson and Stokes equation, the same result holds under the additional condition that the particles only occupy a bounded set in $\R^3$.
In \cite{HoferVelazquez18}, a relaxed version of the method of reflections is then studied which also converges when the particles are distributed everywhere and the capacity density is only bounded.
This relaxed version reads 
\begin{align} \label{eq:v_k.tilde}
	\tilde v_k = (1-\sum_{i \in I} \gamma_i Q_i)^k u,
\end{align}
where the parameters $\gamma_i$ have to be chosen sufficiently small (e.g. $\gamma_i = \gamma e^{-|X_i|}$, $\gamma$ small).
As explained in \cite{HoferVelazquez18} (see in particular Section 2.5 therein), one can view this as resummation or renormalization of the original series \eqref{eq:MOR.series}: 
expanding again the right-hand side of \eqref{eq:v_k.tilde} yields a partial sum corresponding to \eqref{eq:MOR.series}  with coefficients that depend on $\gamma_i$ and $k$ but each coefficient converges to $1$ as $k \to \infty$. 

A corresponding formulation in terms of orthogonal reflections has been used in \cite{Hofer18MeanField} and \cite{NiethammerSchubert19} for the system 
\eqref{eq:Stokes.perforated} - \eqref{eq:torquefree}.\footnote{To be precise, in  \cite{Hofer18MeanField}, a system is studied where particle rotations are neglected but this does not matter a lot for the analysis of the method.} Due to the different boundary conditions, the function $Q_i u$ decays much faster, 
namely roughly like
\begin{align}
	|(Q_i u)|(x) \sim \frac{R_i^3 |eu(X_i)|}{|x - X_i|^2} \qquad \text{as } |x| \to \infty.
\end{align}

Moreover, the gradient $\nabla Q_i u$, which is the relevant quantity for analyzing the convergence of the method of reflections for the system
\eqref{eq:Stokes.perforated} - \eqref{eq:torquefree},  decays like $|x - X_i|^{-3}$.
As a consequence, the convergence properties of the method of reflections for system \eqref{eq:Stokes.perforated} - \eqref{eq:torquefree} are much better than for the problem with Dirichlet boundary conditions. In particular, one expects the convergence to be related to the volume fraction $\phi$
of the particles instead of the capacity density. Indeed, formally, one obtains
\begin{align}
	\sum_{i \in I} \nabla Q_i u(x) \sim  \sum \frac{R_i^3 eu(X_i)}{|x - X_i|^3} \sim \int_{\R^3} \phi(y) 
	\frac{e u(y)}{|x-y|^3} \dd y,
\end{align}
where $\phi(y)$ denotes the local particle volume fraction. Hence, if one assumes that the local particle volume fraction is bounded by some constant $\phi_0$, one expects the first order correction in the method of reflections to scale like $\phi_0$. Similarly, the $k$-th order correction given by the $(k+1)$-st term  on the right-hand side of \eqref{eq:MOR.series} should scale like $\phi_0^k$.

This argument is completely formal. In particular, in order to make it rigorous, one needs to deal with the fact that the decay
$|x - X_i|^{-3}$ is critical for summability.
Therefore, in \cite{Hofer18MeanField} and \cite{NiethammerSchubert19} finite clouds of particles are considered which avoids the summability issue at infinity.
Moreover, since the decay is also critical at zero, instead of the expected convergence for small particle volume fraction $\phi_0$, the critical exponent caused a smallness assumption for
$\phi_0 \log N$, where $N$ is the total number of particles.

In \cite{HillairetWu19,Gerard-VaretHillairet19}, a related method of reflections is studied. 
The problem of the critical exponent is overcome by approximating the sum $\sum \nabla Q_i u$ by an integral which takes the form of a Calderon-Zygmund operator. The method of reflections studied in 
\cite{HillairetWu19,Gerard-VaretHillairet19} replaces the operators $Q_i$ by more explicit operators. This yields a modified method of reflections which produces a sequence $\bar v_k$ which is not convergent to $v$. However, it is shown in \cite{HillairetWu19,Gerard-VaretHillairet19},
that the error $v - \bar v_k$ is sufficiently small in terms of the particle volume fraction for the purpose of analyzing the first and second 
order correction for the effective viscosity of a suspension.

\subsection{Outline of the main results}

In the present paper, we follow the approach in \cite{HoferVelazquez18} for the analysis of the method through studying the operator $L = \sum Q_i$.
In contrast to the corresponding system with Dirichlet boundary conditions studied in \cite{HoferVelazquez18}, we show that $L$ is always a bounded self-adjoint operator on $\dot H^1(\R^3)$ under the mere condition that the particles are non-overlapping, with a bound on $\|L\|$ which does not depend on the particle configuration (see Theorem \ref{th:L.bounded}).
In particular, $\tilde v_k := (1-\gamma L)^k u$ converges to $v$ in $\dot H^1(\R^3)$ for $\gamma < 2 \|L\|$ (see \ref{cor:L^2.no.rate}.
Moreover, we show that $L$ has a spectral gap if there exists $\theta > 1$ such that the balls $B_{\theta R_i}(X_i)$ are disjoint, leading to a quantitative convergence result, Theorem \ref{th:conv.relaxed}.

This convergence results bears the advantage that it holds for a very general class of particle configurations.
However, in many applications, stronger estimates on the rate of convergence are needed, in particular in terms of the particle volume fraction.
Such a result can only be expected to hold under additional separation conditions on the particles, since clusters of close particles slow down the convergence of the method.
We prove quantitative convergence results for the method of reflections in terms of $\phi_0$ as in \eqref{def:phi_0}. 
We prove that for $\phi_0$ sufficiently small, the method of reflections converges in $\dot W^{1,q}(\R^3)$ with rate $\phi_0$ for any
$1 < q < \infty$. 

The quantity $\phi_0$ is an upper bound for the particle volume fraction. The (quite restrictive) assumption that $\phi_0$ scales like the particle volume fraction has been imposed in many related papers (see for instance \cite{DesvillettesGolseRicci08,HillairetMoussaSueur17,Hofer18MeanField,NiethammerSchubert19,Gerard-VaretHillairet19}). 
We refer to \cite{Hillairet2018,Mecherbet18} for results on the convergence under different assumptions.
It remains an important open problem to weaken this assumption in order to treat more general (and in particular random) configurations of particles using 
the method of reflections or an appropriately modified version of it.
As indicated in Section \ref{sec:applications}, without assuming $\phi_0 \sim \phi$, it seems then appropriate to reorganize the series \eqref{eq:MOR.series} in such a way that the $k$-th term contains all the terms involving exactly $k$ particles (so  for example the second term contains not only terms $Q_{i_1}Q_{i_2}$ but also $Q_{i_1}Q_{i_2}Q_{i_1}$, $Q_{i_1}Q_{i_2}Q_{i_1}Q_{i_2}$ etc.).

he present work relies on a combination of the techniques from \cite{HoferVelazquez18,Hofer18MeanField,NiethammerSchubert19,HillairetWu19,Gerard-VaretHillairet19}.
In particular, since the use of estimates for Calderon-Zygmund operators makes it possible to remove the logarithmic divergence in \cite{Hofer18MeanField,NiethammerSchubert19}, we are able to treat the case of infinitely many particles. 
The Calderon-Zygmund estimates are used in a slightly different way than in \cite{HillairetWu19,Gerard-VaretHillairet19}, 
which yields better convergence rates in $\phi_0$. More precisely, we show that the $(k+1)$-st term on the right-hand side of
\eqref{eq:MOR.series} in general scales like $\phi_0^{k-1}$ in the homogeneous Sobolev space $\dot W^{1,q}$ for any $1 < q < \infty$.
As we will discuss after stating this result, Theorem \ref{th:L^q.rate.phi}, these convergence rates seem to be optimal.

For the $L^p$ theory, a crucial ingredient is the uniform a priori estimate  $\|v\|_{\dot W^{1,q}} \leq C \|f\|_{\dot W^{-1,q}}$ for the solution to \eqref{eq:Stokes.perforated}--\eqref{eq:torquefree}  for small $\phi_0$, which is nontrivial for $q \neq 2$. 
Complementary to such perturbative regularity results, large scale stochastic regularity results for \eqref{eq:Stokes.perforated}--\eqref{eq:torquefree} have been established in \cite{DuerinckxGloria20a,DuerinckxGloria21}

To our knowledge, this is the first time that the method of reflections is studied in more general homogeneous Sobolev spaces.
By Sobolev embedding, this allows to obtain estimates in $C^{0,\alpha}(\R^3)$.
Moreover, for bounded clouds of particles, we prove convergence results in $L^\infty(\R^3)$ where we gain a factor in $\phi_0$ compared to $\dot W^{1,q}$ yielding best possible rates discussed in the previous subsection (see Corollary \ref{cor.average}).

These results are expected to be very useful for dynamical problems such as particle sedimentation. In particular, we hope that our results allow to obtain rigorous results for the effective dynamics of particle suspensions at small volume fractions $\phi$ instead of $\phi \to 0$ as in \cite{Hofer18MeanField,Mecherbet18,Mecherbet19, HoeferSchubert21}. 

We point out that for the sake of simplicity, we only treat spherical particles in this paper. This has the benefit that some expressions become explicit, but all the results are expected to hold in for non-spherical particles, too. We refer to \cite{HillairetLacaveWu19} for a treatment of the method of reflections for non-spherical particles.

\medskip

The rest of the paper is organized as follows. In Section \ref{sec:mainResults}, we give precise definition of the method of reflections and the statements of the main results.
In Section \ref{sec:relaxed}, we prove the main results for the relaxed method of reflections.
In Section \ref{sec:L^2}, we prove the main result for the unrelaxed method, Theorem \ref{th:L^q.rate.phi}, in the Hilbert space $\dot H^1(\R^3)$.
Relying on this result, we then generalize this result to $\dot W^{1,q}$ for any $1 < q < \infty$ in Section \ref{sec:L^p}.
Finally, in Section \ref{sec:proofCorollary}, we prove the convergence result in $L^\infty$, Corollary \ref{cor.average}.

\section{Setting and main results}

\label{sec:mainResults}
\subsection{Weak formulation of the problem}

Without loss of generality, we set the viscosity $\mu = 1$ in the rest of the paper.

In order to write problem \eqref{eq:Stokes.perforated} - \eqref{eq:torquefree} in a weak  formulation,
we consider homogeneous Sobolev spaces of divergence free functions which are rigid body motions inside the particles.
More precisely, we denote by $\dot W^{1,q}(\R^3)$, $1 \leq q < \infty$ the homogeneous Sobolev space defined as the closure of $C_c^\infty(\R^3)$ with respect to the semi-norm 
$\| \nabla \cdot \|_{L^q(\R^3)}$. By the Sobolev embedding for $q < 3$, we can identify elements in $\dot W^{1,q}(\R^3)$ with functions in $L^{q^\ast}$, $q^\ast$ being the Sobolev conjugate $q^\ast = (3 q)/(3-q)$. On the other hand, if $q \geq 3$, elements in 
$\dot W^{1,q}(\R^3)$ are equivalence classes of functions which differ by a constant almost everywhere.
We also define 
\begin{align}
	\dot W^{1,q}_\sigma(\R^3) = \{ w \in \dot W^{1,q}(\R^3) : \dv w = 0\}.
\end{align}

Then, we introduce the space 
\begin{align} \label{eq:defW}
	W_q := \left\{ w \in \dot W^{1,q}_\sigma(\R^3) : ew= 0 ~ \text{in} ~ K \right\}.
\end{align}
Then, if $f \in \dot W^{-1,q}(\R^3) := (\dot W^{1,q'}(\R^3))^\ast$ ($q'$ denoting the H\"older conjugate of $q$), the weak formulation of \eqref{eq:Stokes.perforated} - \eqref{eq:torquefree} reads
\begin{equation} \label{eq:weak.from}
\begin{aligned}
	\int_{\R^3} \nabla v : \nabla \varphi &= \langle f, \varphi \rangle 	\qquad \text{for all } \varphi \in W_{q'},  \\
	\dv v & = 0 \quad \text{in } \R^3,\\ 
	e v &= 0 \quad \text{in } K.
\end{aligned}
\end{equation}
This means that we seek a solution $v \in W_q$ which satisfies the first equation in \eqref{eq:weak.from}.

\subsection{The method of reflections in terms of orthogonal projections}

For $q = 2$ we denote for simplicity $W := W_2$.
In this case, the existence and uniqueness of a weak solution to \eqref{eq:weak.from} for any $f \in \dot H^{-1}(\R^3)$ is immediate.
Moreover, we can characterize this solution as follows.
Let $P$ denote the orthogonal projection in $\dot H^1_\sigma(\R^3)$ to $W$.
Then, the solution to \eqref{eq:weak.from} is given by $v = P u$, where $u \in \dot H^1(\R^3)$ is the solution to\footnote{Note that here and in the following we denote by the same letter $p$ different pressure terms. Since the pressure itself is not of interest for our analysis, there will arise no confusion from this slight abuse of notation.}
\begin{align} \label{eq:defU}
	 \Delta u + \nabla p = f, \quad \dv u = 0 \quad \text{in } \R^3.
\end{align}

For the setup of the method of reflections (following \cite{HoferVelazquez18,Hofer18MeanField,NiethammerSchubert19}), we define 
\[
	W_i = \left\{ w \in \Hdot^1_\sigma(\IR^3) : e w= 0 ~ \text{in} ~ B_i \right\}.
\]
Clearly, 
\begin{align}
	W = \bigcap_i W_i.
\end{align}

Let $P_i$ be the orthogonal projection from $\Hdot^1_\sigma(\IR^3) $ to $W_i$ and $Q_i = 1 - P_i$.
We observe (see \cite[Lemma 4.1]{NiethammerSchubert19})
\begin{equation}
	\label{eq:characterizationWPerp}
	\begin{aligned}
	W_i^\perp &= \biggl\{ w \in \Hdot^1_\sigma(\IR^3) \colon -\Delta w + \nabla p = 0  ~ \text{in} ~  \IR^3 \backslash \overline{B_i}, \\
	& \qquad  \int_{\partial B_i}  \sigma[w,p] n =   \int_{\partial B_i}  (x - X_i) \times (\sigma[w,p] n) = 0 \biggr\}.
	\end{aligned}
\end{equation}

The method of reflections can now be stated as follows. As a zero order approximation for the solution $v$ to \eqref{eq:weak.from}, one takes the solution $u$ to \eqref{eq:defU}.
Recall from \eqref{eq:weak.from} that $v$ is a rigid body motion inside of the particles, i.e. $ev = 0$ in $\cup_i B_i$.
Thus, the idea behind the method of reflections is to add functions $w_i$ to $u$ in such a way that $u + w_i$ is a rigid body motion inside of the particle $i$, and still satisfies
the Stokes equations.  Thus, $w_i = -Q_i u$, and we define
\[
	v_1 = (1 - \sum_i Q_i) u.
\]
Clearly, since generally $e w_i \neq 0$ in $B_j$ for $j \neq i$, the function $v_1$ is still not a rigid body motion inside the particles.
Therefore, higher order approximations for $v$ are obtained by repeating this process.
\begin{equation}
	\label{eq:defV_k}
	 v_k  = \left(1 - \sum_i Q_i \right)^k u.
\end{equation}
This defines the $k-th$ order approximation of $v$ through the method of reflections. It is not difficult to see that $v_k$ 
coincides with the series expansion \eqref{eq:MOR.series} cut after the $(k+1)$-st term (see \cite{HoferVelazquez18}).

The elements of $W_i^\perp$ can be interpreted as force dipoles leading to a decay like $|x-X_i|^{-3}$. As discussed in Section \ref{sec:heuristics},
it is this decay that determines the convergence rate of the method of reflections.

\subsection{Main results for the relaxed method}

We want to study the convergence $v_k \to v$. Clearly, convergence cannot be expected in general, since the sequence $v_k$ might be unbounded.
 We therefore need to study the operator 
$L$ defined by
\begin{align} \label{eq:defL}
 L := \sum_i Q_i.
\end{align}
Since the sum is infinite, it is a priori not even clear whether $L$ is a bounded operator on $\dot H^1_\sigma(\R^3)$. 
Indeed, in \cite{HoferVelazquez18}, we studied the corresponding operator when boundary condition in \eqref{eq:Stokes.perforated} - \eqref{eq:torquefree} are replaced by Dirichlet boundary condition. In that case, $L$ generally does not define a bounded operator if the particles are distributed everywhere in the whole space. Indeed, in that case, the corresponding functions $Q_i u$ decay like $|x - X_i|^{-1}$ which is not fast enough to make the sum $\sum Q_i u$ converge.
On the other hand, for the boundary conditions studied here, the functions $Q_i u$ decay faster (like $|x - X_i|^{-2}$)
since they are ``dipole potentials'' meaning
\begin{align}
	\int_{\partial B_i} \sigma[Q_i u, p] n = 0,
\end{align}
because $Q_i u \in W_i^\perp$.

We prove that $L$ is a bounded operator on the mere condition that the particles are pairwise disjoint.

\begin{thm} \label{th:L.bounded}
	Assume that $B_i \cap B_j = \emptyset$ for all $i \neq j \in I$.
	Then, the operator $L = \sum_i Q_i$ is a well defined, bounded, nonnegative and self-adjoint operator on $\dot H^1_{\sigma}(\R^3)$ with $\|L\| \leq C$ for a universal constant $C$.
\end{thm}

\begin{rem}
	Following the proof of this theorem in Section \ref{sec:relaxed}, it is possible to relax the non-overlapping condition.
	More precisely, the theorem remains true, if there exists an $N_0 \in \N$ such that $|\{ i \in I \colon x \in B_i\}| \leq N_0$ for all $x \in \R^3$.
	The norm of $\|L\|$ is then bounded by $C N_0^{1/2}$.
\end{rem}

\medskip

Even though $L$ is a well-defined bounded operator on $\dot H^1_\sigma(\R^3)$, the sequence $v_k$ might be unbounded.
Analogously to \cite{HoferVelazquez18} we therefore also study the convergence of the so called relaxed method of reflections which consists in replacing the sequence $v_k$ by
\begin{align}	\label{eq:defV_k.tilde}
		 \tilde v_k  := \left(1 - \gamma L\right)^k u,
\end{align}
for some small parameter $\gamma >0$.

As a direct consequence from the spectral theorem for bounded self-adjoint operators, the relaxed method of reflections is always convergent, if $\gamma$ is chosen small enough in \eqref{eq:defV_k.tilde}, independently of the particle configuration. More precisely, we have the following corollary.

\begin{cor} \label{cor:L^2.no.rate}
	Assume that $B_i \cap B_j = \emptyset$. There exists a universal constant $\gamma_0 > 0$ such that for all $0 < \gamma < \gamma_0$
	\begin{align}
		(1-\gamma L)^k \to P,
	\end{align}
	pointwise as operators on $\dot H^1_\sigma(\R^3)$, where $P$ is the orthogonal projection from $\dot H^1_\sigma(\R^3)$ to $W$.
	In particular, for all $f \in \dot H^{-1}(\R^3)$, the sequence $\tilde v_k$ defined in \eqref{eq:defV_k.tilde} converges to
	the unique weak solution $v \in \dot H^1(\R^3)$ of problem \eqref{eq:weak.from}.
\end{cor}

\bigskip

For a quantitative convergence result, we impose the following  additional assumption on the particle configuration:
\begin{align} \label{eq:theta.separation}
	\text{there exists} \quad \theta > 1 \quad \text{such that} \quad  B_{\theta R_i}(X_i) \cap B_{\theta R_j}(X_j) = \emptyset \qquad \text{for all } i \neq j.
\end{align}

Under this constraint, we obtain the following result.
\begin{thm} \label{th:conv.relaxed}
Assume that \eqref{eq:theta.separation} holds for some $\theta > 1$. Then, there exist constants  $0<c<C<\infty$ depending only on $\theta$ such that for all $\gamma >0$
	\begin{align}
	\|(1-\gamma L)^k - P\| \leq \max\{1 - c \gamma, |1 - C \gamma|\}^k
\end{align}
\end{thm}
This theorem implies that we may choose $\gamma$ (depending only on $\theta$) such that the convergence $\tilde v_k \to v$ is exponential with a rate that depends only on $\theta$.

\subsection{Main results for the unrelaxed method} 

For convergence of the unrelaxed method of reflections, we need an additional smallness condition. We introduce
\begin{align} \label{def:d_min}
	d_{\min} := \inf_{i \neq j} |X_i - X_j|, && R_{\max} := \sup_{i \in I} R_i.
\end{align}
The smallness needed for the results stated in this subsection is in terms of
\begin{align} \label{def:phi}
	\phi_0 := \frac{R_{\max}^3}{d_{\min}^3}.
\end{align}

In particular, we assume $R_{\max} < \infty$ and $d_{\min} > 0$.
We emphasize that smallness of $\phi_0$ is a stronger condition than the assumption that $\theta$ in \eqref{eq:theta.separation} can be chosen large.
In the case that the particle radii are identical, these conditions are equivalent, though.

We consider the convergence of the method of reflections in the spaces $\dot W^{1,q}(\R^3)$ for $1 < q < \infty$.
Since the operators $Q_i$ are orthogonal projections in $\dot H^1_\sigma(\R^3)$, we first have to obtain a meaningful definition of the method in $\dot W^{1,q}(\R^3)$ by extending these operators and thus $L$ from the dense set $\dot H^1_\sigma(\R^3) \cap \dot W^{1,q}_\sigma(\R^3)$ to $\dot W^{1,q}_\sigma(\R^3)$.

\begin{thm} \label{th:L^q.rate.phi}
	Let $1 < q < \infty$. Then, there exists $\bar \phi_0 > 0$ depending only on $\theta$ and $q$ such that for all 
	$\phi_0 < \bar \phi_0$ the operator $L$ can be extended to a well-defined operator on $\dot W^{1,q}_\sigma(\R^3)$ and for all $f \in \dot W^{-1,q}(\R^3)$
	there exists a unique weak solution $v \in \dot W^{1,q}(\R^3)$ to \eqref{eq:weak.from}. Moreover, for all $k \in \N$
	\begin{align}
		\|(1-L)^k u - v\|_{\dot W^{1,q}(\R^3)} \leq C \|e (1-L)^k u\|_{L^q(\cup_i B_i)} \leq C (C \phi_0)^k \|e u\|_{L^q(\cup_i B_i)},
	\end{align}
	where $u \in \dot W^{1,q}(\R^3)$ is the unique weak solution to \eqref{eq:defU} and $C$ depends only on $q$.
\end{thm}

\begin{rem}
	We point out that this theorem (applied with $k=0$) implies uniform a priori estimates for the problem \eqref{eq:weak.from}, i.e. under the assumptions of the theorem
		\begin{align}
		\|v\|_{\dot W^{1,q}(\R^3)} \leq C \|f\|_{\dot W^{-1,q}(\R^3)}
	\end{align}
	where $C$ depends only on $q$.
\end{rem}

\begin{rem}
	By Morrey's inequality, we see that the  above theorem applied with $q > 3$ in particular implies convergence for the H\"older seminorm $[\cdot]_{C^{0,\alpha}}$,
	$\alpha = 1 - 3/q$. Moreover, if  in addition $f \in \dot W^{-1,r}(\R^3)$, $r < 3$ we can deduce convergence in $C^{0,\alpha}$.
\end{rem}

Theorem \ref{th:L^q.rate.phi} says that we gain a factor $\phi_0$ each time we apply the method of reflections. 
However, it should be emphasized that the estimate only implies that the $k$-th order corrector $L (1-L)^{k-1}$ is of order $\phi_0^{k-1}$.
This $k$-th order corrector coincides with the $(k+1)$-st term on the right hand side of \eqref{eq:MOR.series} and we have argued in Section \ref{sec:heuristics} that this term is expected to be of order $\phi_0^k$. In particular, the above theorem does not imply any smallness on
$u - v$.

Such a smallness cannot hold in general in the space $\dot W^{1,q}(\R^3)$ since 
\begin{align}
	\|u - v\|^q_{\dot W^{1,q}(\R^3)} \geq \|e(u - v)\|^q_{L^q(\cup_i B_i)} = \|e u\|^q_{L^q(\cup_i B_i)}.
\end{align}
For a sufficiently regular function $f$, one might exploit the term $\|e u\|_{L^q(\cup_i B_i)}$ for some smallness in $\phi_0$. At best, though 
one can hope for a factor $\phi_0^{1/q}$. 

\medskip

However, the heuristics that the $k$-th order approximation gives an accuracy $\phi_0^{k+1}$ holds if we look at the function itself instead of the gradient.
For this result we need to make the additional assumption that 
for some $q < 3/2$
\begin{align} \label{eq:summability} \tag{H2}
	\lambda_{q} := \sup_i \sum_{j \neq i}\frac{R_j^3}{|X_i - X_j|^{2q}} < \infty.
\end{align}
We emphasize that this is always fulfilled if all the particles are contained in a bounded domain or if the number density decays sufficiently fast at infinity.
In particular, in these cases $\lambda_q \leq C \phi_0$.

\begin{cor} \label{cor.average}
	Let $1< r < 3 < q < \infty$ and assume that \eqref{eq:summability} holds for $q'$.
	 Then, there exists $\bar \phi_0 > 0$ depending only on $q$ such that for all 
	$\phi_0 < \bar \phi_0$ and for all $f \in \dot W^{-1,q}(\R^3) \cap \dot W^{-1,r}(\R^3)$ and all $k \in \N$
	\begin{align}	\label{eq:L^infty.convergence}
	 	\|(1-L)^k u - v\|_{L^\infty(\R^3)} \leq C (R_{\max}^\alpha + \lambda_{q'}^{1/q'}) (C \phi_0)^k \|e u\|_{L^q(\cup_i B_i)},
	\end{align}
	where $\alpha = 1 - 3/q$ and $C$ depends only on $q$. Moreover,
	\begin{align} \label{eq:convergence.average}
		\sup_{i \in I} \left | \fint_{\partial B_i} (1-L)^k u - v \right| \leq C \lambda_{q'}^{1/q'} (C \phi_0)^k \|e u\|_{L^q(\cup_i B_i)}.
	\end{align}
\end{cor}

\begin{rem}
	\begin{enumerate}[(i)]
		\item To see that the term $R_{\max}^\alpha$ on the right-hand side of \eqref{eq:L^infty.convergence} is needed, we consider the case $k=0$ and look at the difference $u-v$ in $B_i$ for any $i$. The function $v$ is restricted to a rigid body motion inside $B_i$, whereas $u$ is more or less arbitrary.
		Since $u \in C^{0,\alpha}$, the best possible approximation for $u$ in $L^\infty(B_i)$ by a rigid body motion will therefore in general 
		be afflicted with an error $R_{\max}^\alpha$.
	
		\item If one assumes additional regularity on $f$, namely $f \in L^q(\R^3)$, 
		we find $\nabla u \in L^\infty(\R^3)$ and therefore 	$\|e u\|_{L^q(\cup_i B_i)} \leq \phi_0^{1/q}$. 
		Inserting this into \eqref{eq:convergence.average} yields the optimal convergence rate $\phi_0^{k+1}$ provided that $\lambda_{q'} \leq C \phi_0$.
		
		\item The assumption $f \in \dot W^{-1,r}(\R^3)$ in the above corollary is only needed to fix the problem that functions in $\dot W^{1,q}(\R^3)$
		are only defined up to a constant. 
	\end{enumerate}
\end{rem}

\section{Proof of the main results on the relaxed method of reflections}

\label{sec:relaxed}

In this section we prove that the operator $L$ always defines a bounded operator under the mere condition that the particles are disjoint.
Being a sum of orthogonal projections, $L$ is also self-adjoint. Therefore, the convergence of the relaxed method of reflections is a direct consequence of
spectral theory. Moreover, we get quantitative convergence results in the case that the separation condition \eqref{eq:theta.separation} holds.

We first recall from \cite{NiethammerSchubert19} some properties of the subspaces $W_i$ and $W_i^\perp$ (see equation \eqref{eq:defW}) and the corresponding orthogonal projections $P_i$ and $Q_i$.

\begin{lem}[{\cite[Lemma 4.2]{NiethammerSchubert19}}]
\label{lem:projectionInsideParticle}
	Let $w \in W_i^\perp$. Then
	\begin{align}
		\fint_{\partial B_i} w= \fint_{B_i} \frac{\nabla w - (\nabla w)^T }{2} = 0.
	\end{align}
	In particular, for all $w \in  \Hdot^1_\sigma(\IR^3)$ we have
	\begin{align}
		P_i w = V + \omega \times (x - X_i), && V = \fint_{\partial B_i} w, 
		&& \omega = \frac{1}{2} \fint_{B_i} \curl w
	\end{align}
\end{lem}

\begin{lem}[{\cite[Corollary 4.3]{NiethammerSchubert19}}] \label{lem:Q.scalar.product}
	There exists a universal constant $C>0$ such that all $w \in W_i^\perp$,
	\begin{align}
		\|w\|_{\dot H^1(\R^3)} \leq C \|e w\|_{L^2(B_i)}
	\end{align}
\end{lem}
\begin{rem}
	In particular, this Lemma implies that for all $w_1,w_2 \in H^1_\sigma(\R^3)$
	\begin{align}
			\|Q_i w_1\|_{\dot H^1(\R^3)} \leq C \|e w_1\|_{L^2(B_i)} \label{eq:norm.Q_iu},\\
		(w_1,Q_i w_2)_{\Hdot^1(\IR^3)} \leq C \|e w_1\|_{L^2(B_i)} \|e w_2\|_{L^2(B_i)} \label{eq:scalar.product.Q_iu},
	\end{align}
\end{rem}

Theorem \ref{th:L.bounded} easily follows from Lemma \ref{lem:Q.scalar.product}.
\begin{proof}[Proof of Theorem \ref{th:L.bounded}]
	Since $L$ is self-adjoint, we have\footnote{To be precise, we do not know a priori that $L$ is well-defined on $\dot H^1(\R^3)$. However, it is straightforward to resolve this issue by approximation $L$ through taking into account only finitely many particles. We refer to the proof of \cite[Proposition 2.7]{HoferVelazquez18} for the details.}
	\begin{align}
		\|L\|_{\dot H^1_\sigma \to \dot H^1_\sigma} = \sup_{\substack{w \in \dot H^1_\sigma(\R^3) \\ \|w\|_{\dot H^1(\R^3)} = 1}} (Lw,w)_{\Hdot^1(\IR^3)}.   		
	\end{align}
	By \eqref{eq:scalar.product.Q_iu}, we have
	\begin{align}
		(Lw,w)_{\Hdot^1(\IR^3)} = \sum_i (Q_i w , w)_{\Hdot^1(\IR^3)} \leq C \sum_i \|e w \|^2_{L^2(B_i)} \leq C.
	\end{align}
	This finishes the proof. 
\end{proof}

Corollary \ref{cor:L^2.no.rate} is a direct consequence of Theorem \ref{th:L.bounded} and the spectral theorem for bounded self-adjoint operators. We refer to 
\cite[Proposition 2.13]{HoferVelazquez18} for the details.

In order to quantify the convergence rate, we need to prove that $L$ has a spectral gap, i.e., that $L$ is strictly positive on $W^\perp$.
This is proved in the following Lemma under the additional condition \eqref{eq:theta.separation}.
Theorem \ref{th:conv.relaxed} is a direct consequence of Theorem \ref{th:L.bounded} and Lemma \ref{lem:spectral.gap} below. For the details of the proof we refer again to \cite[Proposition 2.13]{HoferVelazquez18}.
\begin{lem} \label{lem:spectral.gap}
	Assume that \eqref{eq:theta.separation} holds. Then, there exists a constant $c > 0$ depending only on $\theta$ from \eqref{eq:theta.separation} such that 
	\begin{align}
		(Lw, w)_{\Hdot^1(\IR^3)} \geq c \|w\|_{\dot H^1(\R^3)}^2 \qquad \text{for all } w \in W^\perp
	\end{align}
\end{lem}

\begin{proof}
	Let  $w \in W^\perp$. We observe that $w$ is the function of minimal norm in 
	\begin{align}
		X_{w} := \{ \varphi \in \dot H^1(\R^3) \colon e \varphi = e w \text{ in } B_i \text{ for all } i \in I\}.
 	\end{align}
 	Indeed, let $\varphi \in X_w$, then $e(w- \varphi) = 0$ in all the particles.
 	Therefore, $w - \varphi \in W$. Thus $w = Q \varphi$, where $Q$ is the orthogonal projection to $W^\perp$. In particular
 	\begin{align}
 		\|\varphi\|_{\dot H^1(\R^3)}^2 = \|w \|_{\dot H^1(\R^3)}^2 + \|\varphi-w\|_{\dot H^1(\R^3)}^2.
 	\end{align}
 	
 	We apply Lemma \ref{lem:extest} below to $Q_i w$ (restricted to $B_i$) to obtain functions $\varphi_i \in \dot H^1(\R^3)$ such that
 	$\supp \varphi_i \subset B_{\theta R_i}(X_i)$,  $\varphi_i = Q_i w$ in $B_i$, and $\|\varphi_i\|_{\dot H^1(\R^3)} \leq C  \|e Q_i w\|_{L^2(B_i)} \leq C   \|Q_i w\|_{\dot H^1(\R^3)}$.
 	Since the balls $B_{\theta R_i}(X_i)$ are disjoint by assumption and $e \varphi_i = e Q w_i = e w_i$ in $B_i$, we deduce
 	that the function $\varphi := \sum \varphi_i$ is an element of $X_w$.
 	Hence, we may estimate
 	\begin{equation}
 	 \label{eq:conclusion.spectral.gap}
	\begin{aligned} 	
 		(Lw ,w)_{\Hdot^1(\IR^3)} &= \sum_i (Q_i w , w) = \sum_i \|Q_i w\|_{\dot H^1(\R^3)}^2  \\
 		&\geq c \sum_i \|\varphi_i \|_{\dot H^1(\R^3)}^2
 		= c \|\varphi \|_{\dot H^1(\R^3)}^2 \geq  c \|w\|_{\dot H^1(\R^3)}^2,
 	\end{aligned}
 	\end{equation}
 	which concludes the proof.
\end{proof}

\begin{rem} \label{rem:equivalent.norm}
	The proof of the Lemma shows that on $W^{\perp}$ an equivalent norm is given by $\|e \cdot\|_{L^2(\cup_i B_i)}$ with
	\begin{align}
	 	\|e w\|_{L^2(\cup_i B_i)} \leq \|w\|_{\dot H^1(\R^3)} \leq C \|e w\|_{L^2(\cup_i B_i)},
	\end{align}
	where the constant $C$ depends only on $\theta$. Indeed, the second inequality follows immediately from \eqref{eq:conclusion.spectral.gap} and \eqref{eq:norm.Q_iu}.
\end{rem}

For the proof of Lemma \ref{lem:spectral.gap}, we used the following lemma.

\begin{lem}[{\cite[Lemma 4.6]{NiethammerSchubert19}}]
	\label{lem:extest}
	Let $1  < q < \infty$, $\theta > 1$, $ r>0 $, and $ x \in \IR^3$,  
	and assume $w \in  W^{1,q}_\sigma(B_r(x))$ satisfies
	\begin{align}
			\fint_{\partial B_r(x)} w= \fint_{B_r(x)} \curl w = 0.
	\end{align}
	Then, there exists an extension $ E w = E_{\theta,r,x,q} w \in W^{1,q}_{\sigma}(\R^3)$ with $\supp E w \subset B_{\theta r}(x)$ such that $E w = w $ in $B_r(x)$ and
	\begin{equation}
		\label{eq:extest}
		\| \nabla E w \|_{L^q(B_{\theta r}(x))} \leq C \| e w \|_{L^q(B_r(x))},
	\end{equation}
	where the constant $ C $ depends only on $\theta$ and $q$.
\end{lem}

\begin{rem} \label{rem:ext.est}
	\begin{enumerate}[(i)]
	\item 
	Lemma 4.6 in \cite{NiethammerSchubert19} only covers the case $q=2$. The general case is analogous, though.
	\item If the condition $\fint_{\partial B_r(x)} \curl w = 0$ is dropped, one gets the same assertion with $\nabla w$ instead of $ew$ on the right-hand side of \eqref{eq:extest}.
	\end{enumerate}
\end{rem}

\section{Proof of Theorem \ref{th:L^q.rate.phi} for \texorpdfstring{$q=2$}{q=2}}

\label{sec:L^2}

 For the sake of a simpler presentation, we will assume in the following that all the radii are identical, i.e., $R_i = R_{\max}$ for all $i \in I$.
It is not difficult to check that all the estimates that we prove are monotone in the particle radii. Therefore, the general case of different particle radii is proved in the same way. We emphasize that this argument is only valid, since the quantity $\phi_0$ from \eqref{def:phi} only depends on $R_{\max}$, whereas the optimal factor $\theta$ in \eqref{eq:theta.separation} is sensitive to the individual particle radii.

Since we assume that all the particles have the same radius $R_{\max}$, it follows from a scaling argument that it suffices to consider the case $R=1$ for the proof of Theorem \ref{th:L^q.rate.phi}. We will therefore assume $R_i=1$ for all $i \in I$ in Sections \ref{sec:L^2} and \ref{sec:L^p}.

\medskip

The convergence results proved in the previous section are based on bounds derived from variational principles, that hold for general particle distributions and the application of the abstract spectral theorem.
In this section, we prove convergence results in $\dot{H}^1(\R^3)$ for the unrelaxed method of reflections for sufficiently dilute particle configurations. For obtaining such convergence results, and to quantify them for small particle volume fraction $\phi_0$, a more detailed study of
the operator $L$ is needed that goes beyond the bounds and the spectral gap found in the previous section.
More precisely, we will prove that the operator $L$ has the following structure.

\begin{prop} \label{pro:char.L}
	There exists an operator $T \colon W^\perp \to W^\perp$ such that $\|T\| \leq C$, where $C$ depends only on $\theta$ and
	\begin{align}  \label{eq:def.T}
		L = (1+\phi_0 T)P_{W^\perp}.
	\end{align}
\end{prop}

Theorem \ref{th:L^q.rate.phi} in the case $q=2$ is a direct consequence of this proposition. For future reference, we give the statement of this result.

\begin{cor} \label{co:L^2.rate.phi}
	There exists a constant $C$ which depends only on $\theta$ such that for all $u \in \dot H^1(\R^3)$
	\begin{align}
		\|(1-L)^k u - P_W u\|_{\dot H^1(\R^3)} \leq (C \phi_0)^k \|P_{W^\perp} u\|_{\dot H^1(\R^3)}
	\end{align}
\end{cor}

As we will see in the next section, for the proof of Theorem \ref{th:L^q.rate.phi} in the general case $1<q<\infty$, we will show an analogous version of Proposition \ref{pro:char.L}.
However, the analogue of the decomposition $\dot H^1_\sigma(\R^3) = W \oplus W^\perp$ for $\dot W^{1,q}_\sigma(\R^3)$ is more subtle. Indeed, as we will see in the next section, to obtain such an analogue decomposition, we will rely on Corollary \ref{co:L^2.rate.phi}.

In order to prove Proposition \ref{pro:char.L}, it is crucial to understand the interactions between the particles in the method. Recall the setup of the method of reflections from \eqref{eq:defV_k} as
\begin{align}
	 v_k  = \left(1 - \sum_i Q_i \right)^k u
\end{align}
and that the iteration of the operator $(1 - \sum_i Q_i)$ is necessary, since $e Q_i w$ generally does not vanish in $B_j$ for $j \neq i$.
In order to quantify this error, we will study $\nabla Q_i w$ in $B_j$.
The following decay estimate has been shown in \cite[Corollary 4.4]{NiethammerSchubert19}
\begin{align}
	|\nabla Q_i w(x)|  \leq C \|e w\|_{L^2(B_i)} \frac{1}{|x - X_i|^3} \qquad  \text{ for all } x \in \R^3 \setminus B_{\theta}(X_i).
\end{align}
Although this estimate is optimal, it is not sufficient in the case when the particles are distributed everywhere in the whole space
due to the critical exponent on the right-hand side.
Therefore, it is not possible to derive the desired estimate by adding the absolute values of all the errors coming from all other particles,
but it can only be achieved by taking into account possible cancellations between these errors.
To be able to exploit such cancellation phenomena, we split the operator $Q_i$ into two parts $Q_i^d$ and $Q_i^q$,
which we call the simple dipole part and the quadrupole part.
As explained in the introduction, $Q_i$ maps to functions that are dipole potentials meaning $\int_{\R^3} -\Delta Q_i u + \nabla p = 0$.
As we will see below, the simple dipoles are explicit, which enables us to study the cancellations using Calderon-Zygmund theory.
The remainder $Q_i^d$ maps to functions with vanishing monopole and dipole moment. Therefore, they decay faster (their gradient like $|x|^{-4}$) yielding summability.
We remark that simple dipoles have also been considered in \cite{NiethammerSchubert19,HillairetWu19,Gerard-VaretHillairet19}  and that the approximation of the method of reflections studied in \cite{HillairetWu19,Gerard-VaretHillairet19} is similar to neglecting the quadrupole parts $Q_i^d$.

\subsection{Simple dipoles}
We denote by $\Sym_0(3)$ the set of all symmetric traceless $3\times 3$ matrices.
We call a function $w \in W_i^\perp$ a simple dipole (potential), if $w$ is affine in $B_i$.
By the characterization of $W_i^\perp$ in Lemma \ref{lem:dipole.char}, this means that there exists $S \in \Sym_0(3)$ such that $w(x) = S (x- X_i)$ in $B_i$.
We denote 
\begin{align}
	V_{i}^d &:= \{ w \in W_i^\perp : \nabla^2 w = 0 \text{ in } B_i \}, \\
	V_{i}^q &:= \{ w \in W_i^\perp : (w,\varphi)_{\dot H^1(\R^3)} = 0 \text { for all } \varphi \in V_{i}^d\}.
\end{align}
We call functions in $V^{q}_i$ quadrupole (potentials).
The functions in $V^d_i$ are solutions to the Stokes system
\begin{align}
	- \Delta w + \nabla p &= 0, \quad \dv w = 0 \qquad \text{in } \R^3 \setminus B_i, \\
	w &= S (x - X_i) \qquad \text{in } B_i.
\end{align}
They can be computed explicitly (see e.g. \cite{GuazzelliMorris12}) to find $w(x) = w_S(x- X_i)$ with
\begin{align} \label{eq:dipole.poential}
	w_S(x) := 
	\begin{cases}
		S x &\quad \text{in } B_i \\
		\frac{5}{2} \frac{ x ( x \cdot S x)}{|x|^5}  + \frac{S x}{|x|^5} - \frac 5 2 \frac {x (S x \cdot x)}{|x|^7} 
		&\quad \text{in } \R^3 \setminus \overline{B_i}. 
	\end{cases}
\end{align}

We now define the operators
\begin{align} \label{def:Q^d}
	Q_i^d := P_{V_i^d} Q_i, && Q_i^q := Q_i - Q_i^d,
\end{align}
where $P_{V_i^d}$ denotes the orthogonal projections to the subspace $V_i^d$.

We can give the following characterization of $Q_i^d$.
 
\begin{lem} \label{lem:dipole.char}
	Let $w \in \dot H^1_{\sigma}(\R^3)$. Then, $Q_i^d w(x) = w_{S}(x - X_i)$, where $w_S$ is defined by \eqref{eq:dipole.poential} and
	\begin{align}
		S = \fint_{B_i} e w.
	\end{align}
\end{lem} 

\begin{proof}
	Without loss of generality, we assume $X_i = 0$. Let $S \in \Sym_0(3)$ and such that $Q_i^d w(x) = w_{S}$.
	Let $T \in \Sym_0(3)$. A direct computation yields
	\begin{align} \label{eq:dipole.charge}
		-\Delta w_T + \nabla p = 5  T  x \mathcal{H}^2 |_{\partial B_1} \quad \text{in } \R^3.
	\end{align}
	Thus, using that $\nu = x$ on $\partial B_i$, we have
	\begin{align}
	0 &= ( Q_i w  - Q_i^d w, w_T )= \langle Q_i w  - Q_i^d w, -\Delta w_T + \nabla p \rangle 
	 =  5 \int_{\partial B_i} (Q_i w  - Q_i^d w) \cdot T \nu  \\ 
	 &= 5 T :  \int_{B_i} \nabla (Q_i w(y)-w_S(y)) \dd y 
	 = 5 T :  \int_{B_i} \nabla Q_i w(y)- S \dd y 
	\end{align}
	Since $T \in \Sym_0(3)$ was arbitrary, we deduce
	\begin{align}
		S = \fint_{B_i}  eQ_i w = \fint_{B_i}  e w,
	\end{align}
	where the last identity is immediate from the definition of $W_i$.
\end{proof}

The following straightforward decay estimates extend Lemma \cite[Lemma 4.7]{NiethammerSchubert19}.
\begin{lem} \label{lem:decay.dipole.quadrupole}
	Let $w \in \dot H^1(\R^3)$. Then, for all $x \in \R^3 \setminus B_{\theta}(X_i)$
	\begin{align}
		\label{eq:dipoleEstimate}
		|Q_i w(x)| \leq C \frac{1}{|x-X_i|^2} \|e w\|_{L^2(B_i)}, &&
		|\nabla Q_i w(x)| \leq C \frac{1}{|x-X_i|^3}  \|e w\|_{L^2(B_i)}, \\
		\label{eq:quadrupoleEstimate}
		|Q_i^q w(x)| \leq C \frac{1}{|x-X_i|^3} \|e w\|_{L^2(B_i)}, &&
		|\nabla Q_i^q w(x)| \leq C \frac{1}{|x-X_i|^4} \|e w\|_{L^2(B_i)}, 
	\end{align}
	where $C$ depends only on $\theta$.
\end{lem}

\begin{proof}	
	We prove the first estimate in \eqref{eq:quadrupoleEstimate}. 
	The other estimates are analogous. 

	Assume without loss of generality $X_i =0$. Since $Q_i^q w \in W_i^\perp$ we know that $f := - \Delta Q_i^q w \in \dot H^{-1}_\sigma(\R^3)$ with $\supp f \subset \overline{B_i}$ such that
	$S f = Q_i^q w$, where $S$ is the solution operator for the Stokes equations in $\R^3$.
	Here, $\dot H^{-1}_\sigma(\R^3)$ denotes the dual of $\dot H^{1}_\sigma(\R^3)$ and $\supp f \subset \overline{B_i}$
	means $\langle f, \varphi \rangle = 0$ for all $\varphi \in \dot H^1_\sigma(\R^3)$ with $\varphi = 0 $ in $B_i$. We have
	\begin{align}
		\|f\|_{\dot H^{-1}_\sigma(\R^3)} = \| Q_i^q w\|_{\dot H^1(\R^3)} \leq \| Q_i w\|_{\dot H^1(\R^3)} \leq C \|e w\|_{L^2(B_i)},
	\end{align}
	where we used Lemma \ref{lem:Q.scalar.product} in the last estimate.
	We denote by $\Phi$ the fundamental solution of the Stokes equations. Then, we claim
	\begin{align}
		|Q^q_i w (x)| &= |(\Phi \ast f)(x)| = \langle f, E(\Phi - (\Phi)_{x,1} - (\nabla \Phi)_{x,1} \cdot (y-x)) \rangle.
	\end{align}
	Here
	\[
		(\Phi)_{x,1} = \fint_{B_{1}(x)} \Phi(y) \dd y,
	\]	
	and similarly for $(\nabla \Phi)_{x,1}$.
	 Moreover, $E(\Phi - (\Phi)_{x,1} - (\nabla \Phi)_{x,1} \cdot (y-x)) \in H^1_{\sigma,0}(B_\theta(x))$ is an
	extension of the restriction of $\Phi - (\Phi)_{x -X_i,\theta} - (\nabla \Phi)_{x,1} \cdot (y-x)$ to 
	$B_{1}(x)$.
	
	Indeed, since $\supp f \subset \overline{B_i}$ the extension does not affect the convolution. Moreover, 
	since $Q^q_i w \in V^q_i$, subtracting the affine function $(\Phi)_{x -X_i,\theta} - (\nabla \Phi)_{x,1} \cdot (y-x)$ has no effect. 
	
By Lemma \ref{lem:extest} and Remark \ref{rem:ext.est}
	we can choose the extension such that 
	\begin{align}
		\|E (\Phi - (\Phi)_{x,1} - (\nabla \Phi)_{x,1} \cdot (y-x))\|_{\dot H^1(\R^3)} \leq C \|\nabla \Phi - (\nabla \Phi)_{x,1} \|_{L^2(B_i)} 
		\leq C \frac{1}{|x|^3},
	\end{align}
		where the last estimate follows from a direct computation for $|x| \geq \theta$. Hence
	\begin{align*}
		|Q^q_i w (x)| &\leq \|f\|_{\dot{H}^{-1}_\sigma(\IR^3)} \|E(\Phi - (\Phi)_{x,1} - (\nabla \Phi)_{x,1} \cdot (y-x))\|_{\dot H^1(\R^3)}  \\
		& \leq C \frac{1}{|x|^3} \|e w\|_{L^2(B_i)}.   \tag*{\qedhere} 
	\end{align*}
\end{proof}

\subsection{Proof of Proposition \ref{pro:char.L}}

\begin{proof}[Proof of Proposition \ref{pro:char.L}]
	Since $L$ is self-adjoint and $\ker L = W$, we have $\operatorname{range} L = W^\perp$. Thus there exists a unique $T \colon W^\perp \to W^\perp$ such that \eqref{eq:def.T} holds. 
	It remains to prove $\|T\| \leq C$, where $C$ depends only on $\theta$.
	By Remark \ref{rem:equivalent.norm}, we can equip $W^\perp$ with the equivalent norm
	\begin{align}
		\|e \cdot \|_{L^2(\cup_i B_i)}.
	\end{align}	 
	Thus, we need to prove
	\begin{align} \label{eq:T_Bounded}
		\|e (1-L) v\|_{L^2(\cup_i B_i)} \leq C \phi_0 \|e v\|_{L^2(\cup_i B_i)} \quad \text{for all } v \in W^\perp.
	\end{align}
	Let $v \in W^\perp$.  We observe that
	\begin{align}
		e (1-L) v = e v - \sum_j e Q_j v = - \sum_{j \neq i} e Q_j v
		 = - \sum_{j \neq i} \left(e Q^d_j v +  e Q^q_j v  \right)
		\qquad \text{in } B_i.
	\end{align}
	where $Q^d_j$ and $Q^q_j v$ are the simple dipole  and the quadrupole part of $Q_j v$, respectively, defined in \eqref{def:Q^d}.
	Thus, it remains to prove that
	\begin{align}
		\sum_i \| \sum_{j \neq i} e Q_j^d v \|^2_{L^2(B_i)} &\leq C \phi_0^2 \|e v\|^2_{L^2(\cup_i B_i)}, \label{eq:est.sum.Q^d} \\
		\sum_i \| \sum_{j \neq i} e Q_j^q v \|^2_{L^2(B_i)} &\leq C \phi_0^{\frac 8 3} \|e v\|^2_{L^2(\cup_i B_i)}. \label{eq:est.sum.Q^q}
	\end{align}		

	We begin with the quadrupole term \eqref{eq:est.sum.Q^q}, which is easier and gives a better estimate than the dipole term \eqref{eq:est.sum.Q^d} thanks to the better estimates for the the quadrupoles in Lemma \ref{lem:decay.dipole.quadrupole}.
	Indeed, by Lemma \ref{lem:decay.dipole.quadrupole}, we have
	\begin{align}
		|e Q_j^q v| \leq \frac{C}{|X_i - X_j|^4} \|e v\|_{L^2(B_j)} \qquad \text{in } B_i \quad \text{for all } j \neq i.
	\end{align}
	Thus, 
	\begin{align}
		\sum_i \| \sum_{j \neq i} e Q^q_j v\|^2_{L^2(B_i)} &\leq C  \sum_i \left( \sum_{j \neq i} \frac{1}{|X_i - X_j|^4} \|e v\|_{L^2(B_j)} \right)^2 \\
		& \leq C \sum_i \sum_{j \neq i} \frac{1}{|X_i - X_j|^4} \|e v\|^2_{L^2(B_j)}   \sum_{k \neq i} \frac{1}{|X_i - X_k|^4}
	\end{align}
	We can estimate the sum in $k$ by an integral to find
	\begin{align}
	\sum_{k \neq i} \frac{1}{|X_i - X_k|^4} \leq  C d_{\min}^{-3} \int_{\R^3 \setminus B_{d_{\min}}(0)} \frac{1}{|x|^4} \dd x \leq C d_{\min}^{-4}
	\end{align}
	After exchanging the sums in $i$ and $j$, we can use the same estimate again, yielding
	\begin{align}
		\sum_i \| \sum_{j \neq i} \nabla Q^q_j v\|^2_{L^2(B_i)} 
		&\leq   C d_{\min}^{-8} \sum_{j} \|e v\|^2_{L^2(B_j)}.
\end{align}
Recalling $\phi_0 = d_{\min}^{-3}$, this finishes the proof of \eqref{eq:est.sum.Q^q}.

It remains to prove \eqref{eq:est.sum.Q^d}. We use the explicit form of the $Q_j^d v$
	 provided by \eqref{eq:dipole.poential}. We introduce 
	 \begin{align}
	 	\mathcal K(x) S  := e\left(\frac{5}{2} \frac{(S x \cdot x) x}{|x|^5}\right) \qquad \text{for all } S \in \Sym_0(3).
	 \end{align}
	 Then, by Lemma \ref{lem:dipole.char} and equation \eqref{eq:dipole.poential}
	 \begin{align}
	 	e Q_j^d v(x) =  \colon \mathcal K(x - X_j) (ev)_j + R(x-X_j).
	 \end{align}
	 where $(ev)_j = \fint_{B_j} ev$ and with an explicit remainder term $R$. We note that $R$ decays sufficiently fast (like $|x|^{-5}$) such that we can apply the same strategy as above for
	 the quadrupole terms  $Q^q_i$.
	 
	 It remains to prove
	 \begin{align}
	 		\sum_i \| \sum_{j \neq i} \mathcal K(x - X_j) (ev)_j \|^2_{L^2(B_i)} &\leq C \phi_0^2 \|e v\|^2_{L^2(\cup_i B_i)},
	 \end{align}
	We introduce the function
	\begin{align}
		g := \sum_j (e v)_{j} |B_{\dmin/4}|^{-1} \1_{B_{\dmin/4}(X_j)}.
	\end{align}
	Then, for all $x \in B_i$
	\begin{equation} \label{eq:sum.to.integral.dipole}
	 \begin{aligned} 
	 		 \sum_{j \neq i} \mathcal K(x - X_j)(ev)_j &= 
	 		 \int_{\R^3 \setminus B_{\dmin/2}(X_i)}  \left(\mathcal K(x - X_j) - \mathcal K(x - y)\right)g(y)  \dd y  \\
	 		 &+ \fint_{B_{\dmin / 4}(X_i)} \int_{\R^3 \setminus B_{\dmin/2}(X_i)} 
	 		 		 \left(\mathcal K(x - y) - \mathcal K(z - y)\right) g(y)  \dd y \dd z \\
	 		&+ \fint_{B_{\dmin / 4}(X_i)} \int_{\R^3 \setminus B_{\dmin/2}(X_i)} \mathcal K(z- y) g(y) \dd y \dd z \\
	 		& =: A^i_1(x) + A^i_2(x) + A^i_3.
	 \end{aligned} 
	 \end{equation}
	 Since for all $x \in B_i$ and $y \in B_{\dmin/2}(X_j)$
	 \begin{align}
	 \left|\mathcal K(x - X_j) - \mathcal K(x - y)\right| \leq \frac{C \dmin}{|X_i - X_j|^4},
	 \end{align}
	the first term on the right-hand side of \eqref{eq:sum.to.integral.dipole} is estimated analogously to \eqref{eq:est.sum.Q^q} above,
	leading to 
	\begin{align}
		\sum_i \|A_1^i\|^2_{L^2(B_i)} \leq C \phi_0^2 \|\nabla v \|_{L^2(\cup_i B_i)}^2
	\end{align} 
	Similarly, $A^i_2$ is controlled by 
		\begin{align}
		\sum_i \|A_2^i\|^2_{L^2(B_i)} \leq C \phi_0^2 \|\nabla v \|_{L^2(\cup_i B_i)}^2
	\end{align} 
	
	In order to estimate $A^i_3$ in  \eqref{eq:sum.to.integral.dipole} (which are constant functions) we notice that
	\begin{align}
		\int_{\R^3 \setminus B_{\dmin/2}(X_i)}  \mathcal K (z- y) g(y) \dd y 
		= (\mathcal K \ast (g-g_i))(z),
	\end{align}
	where 
	\begin{align}
		g_i := (\nabla v)_{B_i} |B_{\dmin/4}|^{-1} \1_{B_{\dmin/4}(X_i)}.
	\end{align}
	Since $\mathcal K$ is a Calderon-Zygmund operator, we deduce
	\begin{align}
		\sum_i \|A_3^i\|_{L^2(B_i)}^2  &= C \sum_i (A_3^i)^2  \\
		& \leq C \dmin^{-3} \sum_i \|\mathcal K \ast (g-g_i)\|_{L^2(B_{\dmin / 4}(X_i))}^2 \\
		&\leq C \dmin^{-3} \sum_i \|\mathcal K \ast g\|_{L^2(B_{\dmin / 4}(X_i))}^2 + C \dmin^{-3} \sum_i \|\mathcal K \ast g_i\|_{L^2(B_{\dmin / 4}(X_i))}^2 \\
		& \leq  C \dmin^{-3}  \|\mathcal K \ast g\|_{L^2(\R^3)}^2 + C \dmin^{-3}  \sum_i \|\mathcal K \ast g_i\|_{L^2(\R^3)}^2 \\
		& \leq C \dmin^{-3}  \|g\|_{L^2(\R^3)}^2 + C \dmin^{-3}  \sum_i \|g_i\|_{L^2(\R^3)}^2 \\
		& \leq C \dmin^{-3}  \|g\|_{L^2(\R^3)}^2  \\
		&\leq C \phi_0^2  \|e v\|^2_{L^2(\cup_i B_i)}
	\end{align}
		This finishes the proof of \eqref{eq:est.sum.Q^d}. The proof of the proposition is complete.
\end{proof}

\section{Proof of Theorem \ref{th:L^q.rate.phi} in the general case}

\label{sec:L^p}

To prove Theorem \ref{th:L^q.rate.phi} for general $1 < q < \infty$, we follow the proof in the case $q = 2$. 
More precisely, we aim for a characterization of the operator $L$ acting on $\dot W^{1,q}(\R^3)$ analogous to Proposition \ref{pro:char.L}.
Clearly, for such a characterization, we first of all need to extend $L = \sum_i Q_i$ to a bounded linear operator on $\dot W^{1,q}(\R^3)$ 
and we need to find a suitable replacement of $W^\perp$.

In view of the characterization of $W_i^\perp$ in \eqref{eq:characterizationWPerp}, we define 
for $1 < q < \infty$ the function spaces
\begin{align}
	V_{i,q} &:= \biggl \{ w \in \dot W_\sigma^{1,q}(\R^3) : -\Delta w + \nabla p = 0  ~ \text{in} ~  \IR^3 \backslash \overline{B_i},  \\
	& \qquad ~ \int_{\partial B_i}  \sigma[w,p] n = 0 =  \int_{\partial B_i}  (x - X_i) \times (\sigma[w,p] n)  \biggr\},\\
	V_q &:=  \biggl \{ w \in \dot W_\sigma^{1,q}(\R^3) : -\Delta w + \nabla p = 0  ~ \text{in} ~  \IR^3 \backslash \overline{B_i}, \\
	& \qquad ~ \int_{\partial B_i}  \sigma[w,p] n = 0 =  \int_{\partial B_i}  (x - X_i) \times (\sigma[w,p] n)  \text{ for all } i \in I\biggr\}.
\end{align}	

For the proof of the analogous version of Proposition \ref{pro:char.L}, we need the following ingredients, which we will make precise below:
\begin{itemize}
	\item The extension of the operators $Q_i$ and $L$ to $\dot W_\sigma^{1,q}(\R^3)$ with corresponding estimates.
	\item The decomposition $\dot W_\sigma^{1,q}(\R^3) = W_q \oplus V_q$, 
	 that $\|e \cdot\|_{L^q(\cup_i B_i)}$ is an equivalent norm on $V_q$ and that the solution $v \in \dot W^{1,q}(\R^3)$ to \eqref{eq:weak.from} is given by $P_{W_q} u$
	 where $u$ is defined as in \eqref{eq:defU}.
	\item The decomposition $Q_i = Q_i^q + Q_i^d$ with decay estimates.
\end{itemize}

The main difficulty is the second point, the decomposition $\dot W_\sigma^{1,q}(\R^3) = W_q \oplus V_q$ with uniform estimates in the particle configurations. This is directly related to a well-posedness result for \eqref{eq:weak.from} and one could try to
 obtain such a result using a classical approach for boundary value problems of elliptic equations. 
However, it seems at least very technical to prove such a result with a uniform bound for all particle configurations with $\phi_0$ sufficiently small.

We therefore take a different  approach  relying on the the method of reflections.
This is the reason why we first studied the convergence in $\dot H^1(\R^3)$ in the previous section.  Interpolating 
the converegence rate in $\dot H^1(\R^3)$ provided by Corollary \ref{cor:L^2.no.rate} with the uniform estimates on $L$ in $\dot W^{1,q}(\R^3)$
that we are going to show, we can prove that the method of reflections converges in $\dot W^{1,q}(\R^3)$ to a solution of 
\eqref{eq:weak.from} provided $\phi_0$ is sufficiently small.
Although we do not obtain optimal convergence rates due to the interpolation, this is enough to deduce uniform bounds on the solution.

\begin{prop}	 \label{pro:L.bounded.L^q}
	Let $1 < q < \infty$. Then the operators $Q_i$ extend to bounded linear operators on $\dot W_\sigma^{1,q}(\R^3)$ such that for all $w \in \dot W_\sigma^{1,q}(\R^3)$
	\begin{align}
	\label{eq:bound.Q_i.q}
		\|Q_i w\|_{\dot W^{1,q}(\R^3)} \leq C \|e w\|_{L^q(B_i)}.
	\end{align}
	where $C$ depends only on $q$.
	Moreover, $L=\sum_i Q_i$ is a well-defined bounded operator on $\dot W^{1,q}_\sigma(\R^3)$ with $W_q \subset \ker L$, $\operatorname{range} L \subset V_q$ and 
	\begin{align} \label{eq:bound.L.q} 
		\|L w\|_{\dot W^{1,q}(\R^3)} \leq C \|e w\|_{L^q(\cup_i B_i)}.
	\end{align}
\end{prop}

\begin{prop} \label{co:decomposition.W_q}
	Let $1 < q < \infty$. Then, there exists $\bar \phi_0 > 0$ depending only on $q$ such that for all 
	$\phi_0 < \bar \phi_0$ and for all $w \in \dot W_\sigma^{1,q}(\R^3)$ there exists a unique decomposition $w = w_1 + w_2$
	with $w_1 \in W_q$, $w_2 \in V_q$.
	Moreover,
	\begin{align}
		\| w_1\|_{\dot W^{1,q}(\R^3)} &\leq C  \| w\|_{\dot W^{1,q}(\R^3)}, \label{eq:est.W_q}\\
		\| w_2\|_{\dot W^{1,q}(\R^3)} &\leq C  \|e w\|_{L^q(\cup_i B_i)} = C \|e w\|_{L^q(\cup_i B_i)} \label{eq:est.V_q},
	\end{align}
	where $C$ depends only on $\theta$ and $q$.
	Furthermore, for all $f \in \dot W^{-1,q}(\R^3)$ there exists a unique weak solution $v \in \dot W^{1,q}(\R^3)$ to \eqref{eq:weak.from}
	which is given by $v = P_{W_q} u$ where $u$ is the unique solution to \eqref{eq:defU} and $P_{W_q}$ denotes the projection in $\dot W^{1,q}_\sigma(\R^3)$ 	to 
	$W_q$.
\end{prop}
\begin{rem}
	For $q \geq 3$, elements $u \in \dot W^{1,q}(\R^3)$ are only determined up to a constant. The uniqueness statement in the theorem above has to be understood in this context.
\end{rem}

\begin{lem} \label{lem:decay.q}
	For $w \in \dot W^{1,q}_\sigma(\R^3)$ let $S:= \fint_{B_i} e w$ and define
\begin{align}
	Q_i^d w(x) := w_{S}(x-X_i),
\end{align}
where $w_{S}$ is given by \eqref{eq:dipole.poential}. Moreover, let 
\begin{align}
	Q_i^q  = Q_i - Q_i^d.
\end{align}
Then
\begin{align} \label{eq:est.dipole.quadrupole.q}
	\|Q_i^q w\|_{\dot W^{1,q}(\R^3)} + \|Q_i^d w\|_{\dot W^{1,q}(\R^3)} \leq C \| e w\|_{L^q(B_i)}.
\end{align}
Moreover, for all $\theta > 1$ and for all $x \in \IR^3 \backslash B_{\theta}(X_i)$,
\begin{align}
		\label{eq:dipoleEstimate.q}
		|Q_i w(x)| \leq C \frac{1}{|x-X_i|^2} \|e w\|_{L^q(B_i)}, &&
		|\nabla Q_i w(x)| \leq C \frac{1}{|x-X_i|^3}  \|e w\|_{L^q(B_i)}, \\
		\label{eq:quadrupoleEstimate.q}
		|Q_i^q w(x)| \leq C \frac{1}{|x-X_i|^3} \|e w\|_{L^q(B_i)}, &&
		|\nabla Q_i^q w(x)| \leq C \frac{1}{|x-X_i|^4} \|e w\|_{L^q(B_i)},
	\end{align}
	where $C$ depends only on $\theta$ and $q$.
\end{lem}

Corresponding to Proposition \ref{pro:char.L}, we deduce the following statement.
\begin{prop} \label{pro:char.L.q}
Let $1 < q < \infty$. Then, there exists $\bar \phi_0 > 0$ depending only on $q$ such that for all 
	$\phi_0 < \bar \phi_0$ there exists an operator $T \colon V_q \to V_q$ such that $\|T\| \leq C$, where $C$ depends only on $q$ and 
	\begin{align}  
		L = (1+\phi_0 T)P_{V_q}.
	\end{align}
\end{prop}
Since the proof of Proposition \ref{pro:char.L.q} is completely analogous to the proof of Proposition \ref{pro:char.L}, we refrain from repeating the details.
Moreover, Theorem \ref{th:L^q.rate.phi} follows directly by combining Proposition \ref{pro:char.L.q} with Proposition \ref{co:decomposition.W_q}.

In the following subsections, we proof Proposition \ref{pro:L.bounded.L^q}, Proposition \ref{co:decomposition.W_q} and Lemma \ref{lem:decay.q}.

\subsection{Proof of Proposition \ref{pro:L.bounded.L^q} and Lemma \ref{lem:decay.q}}

We use the following result taken from \cite[Exercise V.5.1]{Galdi11} 
\begin{thm} \label{th:exist.B^c}
	Let $1<q<3$ and let $g \in  W^{1- 1/q,q}(\partial B)$. Consider the problem
	\begin{equation} \label{eq:Single.L^q}
	\begin{aligned}
		- \Delta w + \nabla p&= 0 \quad \text{in } \R^3 \setminus B, \\
		w &= g \quad \text{on } \partial B.
	\end{aligned}
	\end{equation}
	If
	\begin{align}
		\int_{\partial B} g = 0,
	\end{align}
	then, \eqref{eq:Single.L^q} has a unique weak solution $w \in \dot W_\sigma^{1,q}(\R^3 \setminus B)$ which satisfies
	\begin{align}
		 \|\nabla w \|_{L^p(\R^3 \setminus B)} \leq C \| g\|_{W^{1- 1/q,q}(\partial B)}.
	\end{align}
\end{thm}

\begin{proof}[Proof of Proposition \ref{pro:L.bounded.L^q}]
Using the above theorem, we extend the operators $Q_i$ to $\dot W^{1,q}_\sigma(\R^3)$, $q < 3$, by solving for $w \in \dot W^{1,q}_\sigma(\R^3)$ the problem
	\begin{equation} \label{eq:Q_i.L^q}
	\begin{aligned}
		- \Delta Q_i w + \nabla p&= 0 \quad \text{in } \R^3 \setminus B_i, \\
		Q_i w &= w - \fint_{\partial B_i} w \dd x - \frac{1}{2} \fint_{B_i} \curl w \dd x \times (\cdot - X_i) \quad \text{in }  B_i. \\
	\end{aligned}
	\end{equation}
	For $q = 2$, this definition coincides with the original definition
	of $Q_i$ as an orthogonal projection due to Lemma \ref{lem:projectionInsideParticle}.  Moreover, by the above theorem,
	\begin{align} \label{eq:bound.Q_i.q.1}
		\|Q_i w\|_{\dot W^{1,q}(\R^3)} 
		&\leq C_q \left\|w - \fint_{\partial B_i} w \dd x - \frac{1}{2} \fint_{B_i} \curl w \dd x \times (\cdot - X_i)\right\|_{W^{1,q}(B_i)} \\
		&\leq C_q \|e w\|_{L^q(B_i)},
	\end{align}
	where the last inequality follows from a Korn-Poincar\'e inequality (see e.g. \cite[Lemma 4.4]{NiethammerSchubert19})
	for functions with 
	\begin{align}
		\fint_{\partial B_i} w = \fint_{ B_i} \curl w = 0.
	\end{align}
	
	For $q > 3$, we notice that $w \in \dot W^{1,q}(\R^3)$ implies
	$w|_{B_i} \in H^1(B_i)$. Hence $Q_i w$ is well-defined in $\dot H^1_\sigma(\R^3)$. Moreover, using the decay estimates from Lemma \ref{lem:decay.dipole.quadrupole},
	we have for all $x \in \R^3 \setminus B_2(X_i)$
	\begin{align}
		|(Q_i w)(x)| \leq C \frac{1}{|x-X_i|^2} \|e w\|_{L^q(B_i)}, && |(\nabla Q_i w)(x)| \leq C \frac{1}{|x-X_i|^3} \|e w\|_{L^q(B_i)}
	\end{align}
	Combining this estimate with standard regularity theory for the Stokes equations in the annulus $B_2(X_i) \setminus B_1(X_i)$, we deduce 
	that \eqref{eq:bound.Q_i.q} also holds in this case.
	
	\medskip
	
It remains to show that the operator $L = \sum_i Q_i$ is also bounded on $\dot W^{1,q}_\sigma(\R^3)$ with $W_q \subset \ker L$ and
$\operatorname{range} L \subset V_q$  and that \eqref{eq:bound.L.q} holds.

	Clearly, if $e w = 0$, in $\cup_i B_i$, then $Lw =0$.
	Let $w \in \dot W^{1,q}(\R^3)$ with $\|e w\|_{L^q(\cup_i B_i)}  = 1$. Then, we claim
	\begin{align} \label{eq:zero.charge.q}
	    \int_{\R^3} \nabla (Q_i w) \cdot \nabla \varphi = 0 \qquad \text{for all } \varphi \in  W_{i,q'}, 
	\end{align}	 
	 where $q'$ is the dual H\"older exponent. For $q=2$, this is immediate from the original definition of $Q_i$ as the orthogonal projection to $W_i^\perp$.
	 For $q \neq 2$, \eqref{eq:zero.charge.q} thus follows by density and continuity of $Q_i$ due to \eqref{eq:bound.Q_i.q}.	 
	 In the same way as one obtains the characterization $W^\perp = V_2$ (see \cite[Lemma 4.1]{NiethammerSchubert19}), equation \eqref{eq:zero.charge.q} implies $Q_i w \in V_{i,q}$. Since $V_{i,q} \subset V_q$ this shows $L w \subset V_q$ once we have proved that $L$ is well defined.
	
	Let $ \varphi \in \dot W_\sigma^{1,q'}(\R^3)$ with  $\|\varphi\|_{\dot W^{1,q'}(\R^3)} = 1$. From \eqref{eq:zero.charge.q}, we deduce
	 \begin{align}
	 	   \int_{\R^3} \nabla (Q_i w) \cdot \nabla \varphi = \int_{\R^3} \nabla (Q_i w) \cdot \nabla (Q_i \varphi).
	 \end{align}
	 Therefore, using \eqref{eq:bound.Q_i.q}
	 \begin{align}
	 	\int_{\R^3} \nabla (L w)  \cdot \nabla \varphi = \sum_i \int_{\R^3} \nabla (Q_i w) \cdot \nabla (Q_i \varphi)
	 	&\leq \sum_i \|Q_i w\|_{\dot W^{1,q}(\R^3)} \|Q_i \varphi\|_{\dot W^{1,q'}(\R^3)} \\
	 	&\leq C_q \sum_i \left(\frac{1}{q} \|e w\|^q_{L^q(B_i)} + \frac{1}{q'} \|e \varphi\|^{q'}_{L^{q'}(B_i)}\right) \\
	 	&\leq C_q \left( \frac{1}{q} \|e w\|^q_{L^q(\cup_i B_i)} + \frac{1}{q'} \|\varphi\|^{q'}_{\dot W^{1,q'}(\R^3))} \right)
	 	& \leq C_q,
	 \end{align}
	 where $C_q$ depends only on $q$.
	 
	 This establishes \eqref{eq:bound.L.q} since for all $\psi \in W^{1,q}_\sigma(\R^3)$
	 \begin{align}
	 	\|\psi\|_{\dot W^{1,q}(\R^3)} \leq C_q
	 	\sup_{\substack{ \varphi \in \dot W_\sigma^{1,q'}(\R^3) \\ \|\varphi\|_{\dot W^{1,q'}(\R^3)} = 1}} \int_{\R^3} \nabla \varphi \cdot \nabla \psi.
	 \end{align}
	 This finishes the proof.
\end{proof}

\begin{proof}[Proof of Lemma \ref{lem:decay.q}]
	The estimate 
	\begin{align} \label{eq:est.simple.dipole.q}
		\|Q^d_i w\|_{\dot W^{1,q}(\R^3)} \leq C \| e w\|_{L^q(B_i)}
	\end{align}
	follows directly from the explicit form of $Q^d_i$. Then, \eqref{eq:est.dipole.quadrupole.q} follows from \eqref{eq:bound.Q_i.q} and \eqref{eq:est.simple.dipole.q}.
	The decay estimates \eqref{eq:dipoleEstimate.q} and \eqref{eq:quadrupoleEstimate.q} are proved in the same way as in Lemma \ref{lem:decay.dipole.quadrupole}.
\end{proof}

\subsection{Proof of Proposition \ref{co:decomposition.W_q}}

\begin{proof}[Proof of Proposition \ref{co:decomposition.W_q}]
We begin by proving the existence and uniqueness of solutions to \eqref{eq:weak.from}.
	Assume $q > 2$ (the case $q<2$ is treated analogously).
	Let $r >q$ and $f  \in \dot W^{-1,r}(\R^3) \cap \dot H^{-1}(\R^3)$. As before, we denote by $u \in \dot W^{1,r}(\R^3) \cap \dot H^{1}(\R^3)$ 
	the unique weak solution to \eqref{eq:defU}.
	Employing the method of reflections, we define $v_k := (1-L)^k u$. By Corollary \ref{co:L^2.rate.phi} $v_k \to v \in \dot H^1(\R^3)$ and $v$ solves \eqref{eq:weak.from}.
	Moreover, by Corollary \ref{co:L^2.rate.phi} and Proposition \ref{pro:L.bounded.L^q},
	we have
	\begin{align}
		\|v_k - v_{k+1}\|_{\dot H^1(\R^3)} &= \|L(1-L)^k u\|_{\dot H^1(\R^3)} \leq (C \phi_0)^k \|u\|_{\dot H^1(\R^3)}, \\
		\|v_k - v_{k+1}\|_{\dot W^{1,r}(\R^3)} &= \|L(1-L)^k u\|_{\dot W^{1,r}(\R^3)} \leq C_r^{k+1} \|u\|_{\dot W^{1,r}(\R^3)}.
	\end{align}
	Thus, by interpolation (Riesz-Thorin)
	\begin{align}
	 	\|L(1-L)^k\|_{\dot W^{1,q}(\R^3) \to \dot W^{1,q}(\R^3)} \leq C_r^{(k+1) \lambda} (C \phi_0)^{k (1-\lambda)},
	\end{align}
	where $\lambda$ satisfies $1/q = \lambda/r + (1-\lambda)/2$. Choosing $r = 2 q$, we thus find
	\begin{align}
		\|L(1-L)^k\|_{\dot W^{1,q}(\R^3)\to \dot W^{1,q}(\R^3)} \leq C_q (C_q \phi_0^{\lambda_q})^{k}.
	\end{align}
	In particular, for $\phi_0$  small enough (depending only on $q$) $v_k$ is a Cauchy sequence in $\dot W^{1,q}(\R^3)$,
	implying $v_k \to v$ in $\dot W^{1,q}(\R^3)$ and 
	\begin{align}
		\| v\|_{\dot W^{1,q}(\R^3)} \leq C_q \|u\|_{\dot W^{1,q}(\R^3)}\leq C_q  \|f\|_{\dot W^{-1,q}(\R^3)}.
	\end{align}
	By density, the same result holds for any $f \in \dot W^{-1,q}(\R^3)$.
	
	It remains to prove uniqueness. By linearity it is enough to show that if $v \in  \dot W^{1,q}(\R^3)$ satisfies \eqref{eq:weak.from}
	with $f = 0$, 
	then $v=0$.
	Let $g \in \dot W^{-1,q'}(\R^3)$ and let $w \in \dot W^{1,q'}(\R^3)$ be a solution to \eqref{eq:weak.from} with right-hand side $g$.
	Then, 
	\begin{align}
		\langle g, v\rangle = \int_{\R^3} \nabla w :\nabla v = \langle 0, w \rangle = 0.
	\end{align}
	Since $g \in \dot W^{-1,q}(\R^3)$ was arbitrary, this implies $v = 0$.

\medskip

	We now turn to the decomposition $\dot W_\sigma^{1,q}(\R^3) = W_q \oplus V_q$.
	Let $w \in \dot W_\sigma^{1,q}(\R^3)$. Then  $w = w_1+ w_2$ with $w_1 \in W_q$ and $w_2 \in V_q$ if and only if $w_1$ solves \eqref{eq:weak.from} for $f = - \Delta w$. 
	As we have just proved, this problem has a unique solution, which establishes $\dot W^{1,q}(\R^3) = V_q \oplus W_q$ as well as the
	estimate \eqref{eq:est.W_q}. In particular, this shows the assertion $v = P_{W_q} u$ from the proposition.
	
	 It remains to prove \eqref{eq:est.V_q}. To this end, we observe that for the decomposition $w = w_1 +w_2$, the function $w_2$
	only depends on $e w|_{\cup_i B_i}$. In particular,
	\begin{align}
		\| w_2\|_{\dot W^{1,q}(\R^3)} &\leq C_q  \| \tilde w\|_{\dot W^{1,q}(\R^3)}
	\end{align}
	for any function $\tilde w \in \dot W_\sigma^{1,q}(\R^3)$ such that $e \tilde w |_{\cup_i B_i} =  ew |_{\cup_i B_i}$.
	By Lemma \ref{rem:ext.est}, such a function $\tilde w$ exists with
	\begin{align}
		 \| \tilde w\|_{\dot W^{1,q}(\R^3)} \leq C_q\|e w\|_{L^q(\cup_i B_i)}.
	\end{align}
	This finishes the proof.
\end{proof}

\begin{rem}
 It is possible to rely on Theorem \ref{th:conv.relaxed} instead of Corollary \ref{co:L^2.rate.phi} for the interpolation in the method of reflections in the proof above 
	to show the result of Proposition \ref{co:decomposition.W_q} for $p$ close to $2$ without the smallness condition on $\phi_0$. 
	
	More precisely, the following holds. For all $\theta > 2$ there exists $\delta > 0$ such that
	for all $q \in (2-\delta, 2 + \delta)$ and all $w \in \dot W^{1,q}_\sigma(\R^3)$ there exists a unique
	decomposition $w = w_1 + w_2$, with $w_1 \in W_q$ and $w_2 \in V_q$. Moreover, $w_1$ and $w_2$ satisfy \eqref{eq:est.W_q} and \eqref{eq:est.V_q}, respectively, with a constant $C$ which depends only on $\theta$.  
\end{rem}

\section{Proof of Corollary \ref{cor.average}}

\label{sec:proofCorollary}

\begin{proof}[Proof of Corollary \ref{cor.average}]
	We claim that it suffices to prove that for all $w \in \dot W^{1,q}_\sigma(\R^3)$
	\begin{align} \label{eq:cor.claim.iteration}
		\|(1-L) w - w \|_{L^\infty(\R^3)} \leq C (R_{\max}^\alpha + \lambda_{q'}^{1/q'}) \|e w\|_{L^q(\cup_i B_i)}.
	\end{align}
	Indeed, assume that \eqref{eq:cor.claim.iteration} holds. Then, by Theorem \ref{th:L^q.rate.phi},
	\begin{align}
		 \|(1-L)^k u - v\|_{L^\infty(\R^3)} 
		&\leq \sum_{n=k}^\infty  \|(1-L)^n u - (1-L)^{n+1} u\|_{L^\infty(\R^3)}   \\
		&\leq C (R_{\max}^\alpha + \lambda_{q'}^{1/q'}) \sum_{n=k}^\infty \|e (1-L)^n u \|_{L^q(\cup_i B_i)}\\
		&\leq C (R_{\max}^\alpha + \lambda_{q'}^{1/q'}) \|e u \|_{L^q(\cup_i B_i)} \sum_{n=k}^\infty (C \phi_0)^n \\
		& \leq C (R_{\max}^\alpha + \lambda_{q'}^{1/q'}) (C \phi_0)^{k}  \|e u \|_{L^q(\cup_i B_i)}
	\end{align}
	provided $\phi_0$ is sufficiently small.
	
	\medskip
	
	It remains to prove \eqref{eq:cor.claim.iteration}. 
	Let $x \in \R^3$ and let $i_x \in I$ be the minimizer of $|X_i - x|$. (We can disregard the set, where the minimizer is not unique because it is a nullset.) It is easy to generalize the proof of Lemma \ref{lem:decay.q} for an arbitrary particle radius $R_j>0$ to see that
	for all $j \neq i_x$
	\begin{align} \label{eq:far.balls}
		|Q_j w (x)| \leq \frac{C R_{j}^{3/ q'}}{|X_i - X_j|^2} \|e w \|_{L^q( B_j)},
	\end{align}
	Moreover, by the maximum modulus estimate for the Stokes equations, (see \cite[Theorem 6.1]{MaremontiRussoStarita99})
	\begin{align}
		|Q_{i_x}w(x)| \leq C \|Q_{i_x} w \|_{L^\infty(B_{i_x})} \leq C R_{\max}^{\alpha} [Q_{i_x}w]_{C^{0,\alpha}(B_{i_x})}.
	\end{align}
	where we used that $\int_{\partial B_{i_x}} Q_{i_x} w = 0$. From Sobolev embedding it follows
	\begin{align} \label{eq:close.ball}
		|Q_{i_x}w(x)| \leq C R_{\max}^\alpha \|Q_{i_x}w\|_{W^{1,q}(B_{i_x})} \leq C R_{\max}^\alpha \|e w\|_{L^q(B_{i_x})},
	\end{align}
	where we used the Korn-Poincar\'e inequality in the last estimate in the same way as in \eqref{eq:bound.Q_i.q.1}.
	
	Combining \eqref{eq:far.balls} and \eqref{eq:close.ball} yields
	\begin{align}
		|(1-L) w - w)(x)| &\leq |Q_{i_x} w (x)| + \sum_{j \neq i_x} |Q_j w(x)| \\
		&\leq C  R_{\max}^\alpha \|e w\|_{L^q(B_{i_x})} +  C \sum_{j \neq i} \frac{R_{j}^{3 /q'}}{|X_i - X_j|^2} \|e w\|_{L^q(B_j)} \\
		& \leq C  R_{\max}^\alpha \|e w\|_{L^q(B_{i_x})} + C \left(\sum_{j \neq i} \frac{R_j^3}{|X_i - X_j|^{2 q'} } \right)^{1/q'} \left( \sum_{j} \|e w\|_{L^q(B_j)}^q \right)^{1/q} \\
		& \leq C \left(R_{\max}^\alpha + \lambda_{q'}^{1/q'} \right) \|e w\|_{L^q(\cup_j B_j)}
	\end{align}
	for $2 q' < 3$, i.e. $q > 3$, where we used \eqref{eq:summability}.
	This finishes the proof of \eqref{eq:L^infty.convergence}

Estimate \eqref{eq:convergence.average} is proven analogously. In this case, the term $R_{\max}^\alpha$ is not needed, because $\fint_{\partial B_i} Q_i w = 0$.
\end{proof}

\section*{Acknowledgement}

The author wants to thank Richard Schubert for helpful discussions.
The author has been supported by the Deutsche Forschungsgemeinschaft (DFG, German Research Foundation) 
through the collaborative research center ``The Mathematics of Emerging Effects'' (CRC 1060, Projekt-ID 211504053) 
and the Hausdorff Center for Mathematics (GZ 2047/1, Projekt-ID 390685813).

\section*{Conflict of interest}

The author declares no conflict of interest.

\printbibliography

\end{document}